\documentclass[11pt]{article}
\usepackage{verbatim}
\usepackage{geometry} 
\usepackage{tikz}
\usepackage{amsmath, amsfonts, amsthm}
\usepackage{amssymb}
\newtheorem{theorem}{Theorem}[section]
\newtheorem{lemma}[theorem]{Lemma}

\newtheorem{corollary}[theorem]{Corollary}
\newtheorem{algorithm}{Algorithm}

\newenvironment{definition}[1][Definition.]{\begin{trivlist}
\item[\hskip \labelsep {\bfseries #1}]}{\end{trivlist}}

\newenvironment{remark}[1][Remark:]{\begin{trivlist}
\item[\hskip \labelsep {\bfseries #1}]}{\end{trivlist}}

\title{The Intersection of Two Fermat Hypersurfaces in $\mathbb{P}^3$
via Computation of Quotient Curves}
\author{Vijaykumar Singh\thanks{Research supported by Claude Shannon Institute,
Science Foundation Ireland Grant 06/MI/006.}
\ and
Gary McGuire\thanks{Research supported by Claude Shannon Institute,
Science Foundation Ireland Grant 06/MI/006.}
 \\
School of Mathematical Sciences\\
University College Dublin\\
Ireland\\
Email: \texttt{vijaykumar.singh@ucdconnect.ie}, \texttt{gary.mcguire@ucd.ie}}

\begin{document}

\maketitle

\begin{abstract}
 We study the  intersection of two particular Fermat hypersurfaces  in $\mathbb{P}^3$
 over a finite field. 
 Using the Kani-Rosen decomposition we study arithmetic properties of this curve in terms of its quotients.  Explicit computation of the quotients is done using a Gr\"obner basis algorithm.  We also study the $p$-rank, zeta function, and number of rational points,
 of the modulo $p$ reduction of the curve.
 We show that the Jacobian of the genus 2 quotient is $(4,4)$-split.
\end{abstract}

\section{Introduction}

Let $\mathbb{F}_q$ denote the finite field with $q$ elements, where $q=p^n$ and $p$ is a prime.
The equations 
$x^a+y^a+z^a+w^a=0$, $x^b+y^b+z^b+w^b=0$
each define a Fermat hypersurface in $\mathbb{P}^3(\mathbb{F}_q)$, and their
intersection defines a curve.
This article studies some arithmetic properties of this curve in the 
particular case $a=2$, $b=3$, and $p\geq 5$.
We therefore study the curve defined by the  equations
\[
x^2+y^2+z^2+w^2=0, \ \ \ x^3+y^3+z^3+w^3=0,
\]
over $\mathbb{Q}$ (and $\overline{\mathbb{Q}}$), 
by studying some of its quotient curves, 
and consider the reduction modulo $p$.
It would be nice to have a formula for the number of rational points over $\mathbb{F}_p$,
but our investigations show that such a formula is unlikely to exist.

Previous work on these curves has mostly been done when $p=2$, because of connections to cyclic codes in algebraic coding theory.  The case $a=1$ and $b$ arbitrary (when $p=2$) was studied in 
\cite{JMW}.  The case $a=1$ is somewhat different to the case $a>1$ because when $a=1$ the
curve is a plane curve and its equation is easily obtained.
Cases $a=3, b=5$ and $a=1, b=11$ are studied in \cite{Gary},
where the zeta function was computed and applied.
The two above equations can be considered as a system of diagonal equations. 
Such systems have been studied before. Wolfmann  \cite{Wolfmann} and Knapp  \cite{Knapp} study a system with terms having the same exponents but different coefficients,
 over finite fields and local fields respectively. 

  In Section \ref{bg} we present some background needed for the paper.
  In Section \ref{alg1} we give some relevant invariant theory, which is needed to provide the Kani-Rosen decomposition,
  and our algorithm to compute the Gr\"obner basis for a quotient curve, which may be new.
  Section \ref{firstthings} calculates the genus, proves nonsingularity, and computes
  the quotients involved in the Kani-Rosen decomposition.
  Section \ref{g1quots} studies the two genus 1 quotients, and Section 
  \ref{g2quots} studies the genus 2 quotient.
  The genus 2 quotient exhibits some interesting behaviour, having a Jacobian which we prove 
  is a $(4,4)$-split.  These two sections work over $\overline{\mathbb{Q}}$.
 Section \ref{reductions} considers the good reductions of our curve modulo primes,
 and the L-polynomials.
Section \ref{pts} deals with the number of points on the curves over $\mathbb{F}_q$, 
and shows some data from computations.

\bigskip
\noindent\emph{Acknowledgements.}
We would like to thank Nils Bruin and Christophe Ritzenthaler 
for their helpful comments,
and Alexey Zaytsev 
for many helpful discussions.
\section{Background}\label{bg}

In this section we  present a summary of the background needed for this paper.

\subsection{Algebraic Geometry}
Let $k$ be a field, $\mathbb{A}^n(k)$ be affine $n$-space. For $S \subset k[x_1,\dots, x_n]$ define 
\begin{displaymath}
V(S)=\{x \in \mathbb{A}^n(k)|f(x)=0 \ \ \forall f \in S\}.
\end{displaymath}
If $S$ consists of a single (nonconstant) polynomial then $V(S)$ is called a hypersurface.
A subset $V$ of $\mathbb{A}^n(k)$ is called an affine algebraic set if $V=V(S)$ for some $S \subset k[x_1,\dots, x_n]$. 
Then all polynomials in $k[x_1,\dots, x_n]$ that vanish on $V$ forms an ideal $I(V)$ called ideal of V. The affine coordinate ring $k[V]$ of $V$ is quotient ring $k[x_1,\dots, x_n]/I(V)$. If $V$ is irreducible or equivalently $I(V)$ is a prime ideal then $V$ is called affine algebraic variety and the quotient of integral domain $k[x_1,\dots, x_n]/I(V)$  is called function field of $V$  denoted by $K(V)$. If we allow only the homogenous polynomials in $S$ we have a projective algebraic set and hence a projective variety. If for a variety $C$, the function field $k(C)$ has a transcendental degree $1$ over $k$, then $C$ is called an algebraic curve. 

The genus of a curve is defined using the Riemann-Roch theorem.
A non-singular curve of genus 1 is called an elliptic curve. A curve $X$ of genus $\geq 2 $ is called hyperelliptic if there is a map $f:X\rightarrow \mathbb{P}^1$ of degree $2$.

An algebraic group  is a variety $V$ together with a morphism $\oplus:V\times V \rightarrow V$  such that the set of points of $V$ with the operation given by $\oplus$ is a group. A projective algebraic group is called an abelian variety. 

Let $A$ and $B$ be two abelian varieties. The morphisms from $A$ to $B$, which is also a group homomorphism forms a $\mathbb{Z}$-module is denoted by $Hom(A,B)$. 
An element  $\phi \in Hom(A,B)$ is called an \textit{isogeny} if $ker(\phi)$ is finite and $Im(\phi)=B$.

\subsection{Curves}

Let $C$ be an algebraic curve defined over $k$. 
 Recall that the \textit{Jacobian variety}  $J_C$ is an abelian variety with the group structure corresponding the quotient group $Pic^0 (C)$, the degree 0 divisors modulo the
 principal divisors.

 The \textit{zeta function}  of a curve  $C/\mathbb{F}_{q}$  is given by
  \begin{displaymath}
Z(t,C):=exp\left \{ \sum^{\infty  }_{m=1} N_m \frac{t^m}{m}\right\} 
  \end{displaymath}
  where $N_m$ is the number $\mathbb{F}_q^m$-rational points.
  It was shown by Artin and Schmidt that $Z(t,C)$ can be written in the form
  \[
   \frac{L(t)}{(1-qT)(1-T)}
   \]
  where $L(t) \in \mathbb{Z}[t]$ (called the $L$-polynomial) is of degree $2g$, where $g$ is the genus of the curve $C$.  
    For any abelian variety $A=J_C$, the characteristic polynomial of the Frobenius endomorphism (acting on the $l$-adic Tate module of $A$), will be denoted $f_A(t)$.
It is independent of $l$ and has coefficients in $\mathbb{Z}$. In fact, Weil proved that
$f_A(t)$ is the reciprocal polynomial of the $L$-polynomial of $C$.  

We will use the following fundamental result.
 
 \begin{theorem}(Tate)\label{tate} If $A$ and $B$ are the abelian varieties defined over $\mathbb{F}_q$. Then $A$ is $\mathbb{F}_q$- isogenous to abelian subvariety of $B$ if and only if $f_A(t)$ divides  $f_B(t)$ in $\mathbb{Q}[t]$. In particular, $f_A(t)= f_B(t)$ if and only if $A$ and $B$ are $\mathbb{F}_q$-isogenous.
\end{theorem}

Given an abelian variety A defined over $k$, the $p$-rank of $A$ is defined by
\begin{displaymath}
                         r_p (A) = dim_{\mathbb{F}_p} A(\overline{k})[p],
\end{displaymath}
where $A(\overline{k})[p]$ is the subgroup of $p$-torsion points over the algebraic closure. We have $0 \leq r_p (A) \leq dim(A)$. The number $r_p(A)$ is invariant under isogenies over $k$, and satisfies $r_p (A_{1} \times A_{2} ) = r_p (A_{1}) + r_p (A_{2} )$. An elliptic curve over $k$ is called ordinary if its $p$-rank is $1$ and is called supersingular if its $p$-rank is $0$. In fact,  an elliptic curve is supersingular if and only if $p|a_{1}$ where $f(t)= x^2+a_{1}x+q$  is its characteristic polynomial. An abelian variety A of dimension $g$ over $k$ is called ordinary if its $p$-rank is $g$ and supersingular if it is $\overline{k}$-isogenous to a power of a supersingular elliptic curve. A curve $C$ over $\mathbb{F}_{q}$ is called supersingular if the Jacobian $J_{C}(\mathbb{F}_{q})$ of $C$ is supersingular.

We also will compare our curves to the following famous bound.

\begin{theorem}(Hasse-Weil- Serre Bound) If $C$ is a curve of genus $g$ over $\mathbb{F}_q$, then
\begin{displaymath}  |  \# C(\mathbb{F}_q)-( q + 1)| \leq \lfloor 2 \sqrt{q} \rfloor
\end{displaymath}
\end{theorem}
The curve $C$ is said to have defect $\Delta$ if 
\begin{displaymath}    \# C(\mathbb{F}_q)= q + 1 + \lfloor 2 \sqrt{q} \rfloor-\Delta.
\end{displaymath}

\subsection{Kani-Rosen Decomposition}

  If $G \leq Aut(C)$ is a automorphism group of a curve $C$, then any $H \leq G$ defines an idempotent $\epsilon_H \in \mathbb{Q}[G]$  by \\
   \begin{displaymath}
   \epsilon_H := \frac{1}{\left|H\right|}\sum_{h \in H}h.
  \end{displaymath}
  \begin{definition} For  $\epsilon_1 ,  \epsilon_2  \in End_0(J_C):=  End(J_C) \otimes_\mathbb{Z} \overline{\mathbb{Q}} $, we say $\epsilon_1 \sim  \epsilon_2$ if for all $\overline{\mathbb{Q}}$-characters $\chi$ of $End_0(J_C)$. we have $\chi(\epsilon_1 )= \chi(\epsilon_2)$.
\end{definition}
  Also $\epsilon_H(J_C) \sim J_{C/H}$.\\

   \begin{theorem} \emph{(Kani-Rosen)}\label{Kani} Given a curve $C$, let $G \leq Aut(C)$ be a finite
group. Given   $\epsilon_1,\ldots \epsilon_n,\epsilon^{'}_1,\ldots \epsilon^{'}_m \in End_{0}(J_C) $  the idempotent relation
\begin{displaymath}
 \epsilon_1+\ldots +\epsilon_n  \sim \epsilon^{'}_1 +\ldots \epsilon^{'}_m
\end{displaymath}
holds in $ End_{0}(J_C) $ if and only  if we have following isogeny relation
\begin{displaymath}
 \epsilon_1(J_C)+\ldots +\epsilon_n(J_C)  \sim \epsilon^{'}_1(J_C) +\ldots \epsilon^{'}_m (J_C).
\end{displaymath}
  \end{theorem}

\section{Invariant theory and Quotient Curves}\label{alg1}

Here we present our algorithm for computing quotient curves of our given curve, using
Gr\"obner bases.

 Let $K[X]:=K[x_{1},\ldots x_{n}]$ be a polynomial ring in n variables over the field $K$. Let $\Gamma$ be a sub-group of the group $GL_{n}(K)$ of invertible $n\times n$-matrices with entries from $K$. Then\\
\begin{displaymath}
 K[x_{1},\ldots x_{n}]^{\Gamma}:=\{f\in  K[x_{1},\ldots x_{n}]| \ f=f\circ\pi
 \ \textrm{ for all }  \pi \in \Gamma \}.
\end{displaymath} 
is a subring of $K[X]$.
The Reynolds operator  $R_{\Gamma}$ of the group $\Gamma$ is defined as\\ 
\begin{displaymath}
R_{\Gamma}:K[X]\rightarrow K[X]^{\Gamma},~~~~~~~~~~~ R_{\Gamma}(f):=\frac{1}{|\Gamma|}\sum_{\pi \in \Gamma}f \circ\pi 
\end{displaymath}

\begin{theorem}({\emph Hilbert's finiteness theorem}) The invariant ring $(K[X])^\Gamma$ of a finite matrix group $\Gamma \in GL_{n}(K)$ is finitely generated. 
\end{theorem}

\begin{theorem}(\emph{Noether's degree bound}) The invariant ring $K[x]^{\Gamma}$ of finite matrix group $\Gamma$ has an algebra basis consisting of atmost
$n+|\Gamma| \choose n$  invariants whose degree is bounded above by the group order $|\Gamma|$.
\end{theorem}
To help find the exact number of number of invariants we have following results\\
\begin{theorem}(Molien) The Hilbert series of the invariant ring of  $K[x]^{\Gamma}$ equals
\begin{displaymath}
\Phi_{\Gamma}(z)=\frac{1}{| \Gamma|} \sum_{\pi \in \Gamma}\frac{1}{det(id-z \pi)}.
\end{displaymath}
\end{theorem}

The following result may or may not be new; we have not found a reference.

\begin{theorem}Let  $char(K)\nmid |\Gamma|$ and $I\subset K[X]$ be ideal fixed by $\Gamma$.  Then\\
\begin{displaymath}
\frac{K[X]^{\Gamma}}{ I\cap K[X]^{\Gamma}} \cong (\frac{K[X]}{I})^{\Gamma}.
\end{displaymath}
\end{theorem}
\begin{proof}
Define a map
\begin{displaymath} \Pi: K[X]^{\Gamma}\rightarrow (\frac{K[X]}{I})^{\Gamma},~~~~~~~ 
                    f\rightarrow f+I.
\end{displaymath}
We claim that this map is surjective.   
For all $\pi \in \Gamma$ we have $(f+I)^{\pi}:=f^{\pi}+I$.
Let $f+I\in (\frac{K[X]}{I})^{\Gamma}$. This implies $f^{\pi}+I= f+I$ for all $\pi \in \Gamma$.
As $I$ is fixed by $\Gamma$ and $\Gamma$ is finite, summing over all the elements of $\Gamma$ we get 
$|\Gamma|f+I= \sum_{\pi \in \Gamma}f^{\pi}+I$.
Since $|\Gamma|\nmid char K$, we have $f+I=R_{\Gamma}(f)+I$.
As $R_{\Gamma}(f)\in K[X]^{\Gamma}$ the map defined above is surjective.
For kernel, $\Pi(f)=I$ which is possible if and only $f \in I\cap K[X]^{\Gamma}$, which proves the assertion.
\end{proof}
\begin{theorem} Let $K(V):=Quot(K[V])$ be a rational function field on vector space $V$. Then $K(V)^{\Gamma} =Quot(K[V])^{\Gamma}$.
\end{theorem}
\begin{proof}
See \cite{Kemp} 3.7.8.
\end{proof}
\begin{definition}
Let $C$ be a curve and let $\Gamma$ be a 
subgroup of the automorphism group $Aut(C)$. Then quotient curve $C/\Gamma$ is the curve whose function field is given by\\
\begin{displaymath}
K(C/\Gamma):=\{f\in K(C)| f=f\circ\pi \ \textrm{ for all } \pi \in \Gamma \}.
\end{displaymath}
\end{definition}

Here is our algorithm for computing a quotient curve, which we will make use of in this article.

\begin{algorithm}\label{A1} 
\textbf{Input:}
A curve $C$ and $\Gamma$ a subgroup of the automorphism group $Aut(C)$.
\textbf{Output:}
The set of polynomials defining the quotient variety $C/\Gamma$.
\begin{enumerate}
\item Compute the set of polynomials $F$ invariant under $\Gamma$ defining the variety $C$.
\item Compute fundamental set of invariants $\{I_{1}(x),\ldots ,I_{r}(x) \}$.
\item Compute Gr\"obner basis $G_{0}$ for the ideal generated by $\{I_{1}(x)-y_{1},\ldots ,I_{r}(x)-y_{r} \}$ in $K[x_{1},\ldots,x_{n},y_{1},\ldots ,y_{r}]$.
\item Compute Gr\"obner basis $G_{1} $ of $F \cup G_{0}$.
\item Compute $H:=G_{1}\cap K[y_1,\ldots y_r]$. Then $V(H)=C/ \Gamma$.
\end{enumerate}
\end{algorithm}

 It is clear from the definition of the curve we are studying, which is defined by
 $ x^2+y^2+z^2+w^2=0,$    $ x^3+y^3+z^3+w^3=0$, that the symmetric group $S_4$ is 
a subgroup of the automorphism group of $C$.
We are going to see later that the following quotient curves are important:
\begin{eqnarray*}
&C/(1,2)\\
&C/(1,2,3)\\
&C/(1,2,3,4).
\end{eqnarray*}
Here we identify $1$ with $x$, 2 with $y$, and so on, so the permutation $(1,2,3)$
denotes the map that sends $x$ to $y$, and $y$ to $z$, and $z$ to $x$, and fixes  $w$.

The next sections study these quotients, using Algorithm 1.

 \section{ First Things}\label{firstthings}

For the rest of this paper, $C$ will denote the curve in $\mathbb{P}^3(\overline{\mathbb{Q}})$ defined by
the two equations
\begin{displaymath}
   \begin{array}{c}
                     F_1(x,y,z,w)=     x^2+y^2+z^2+w^2=0 \\
                     F_2(x,y,z,w)=     x^3+y^3+z^3+w^3=0. \\
                          \end{array} 
\end{displaymath}

  \begin{lemma} The curve $C$ is nonsingular over $\overline{\mathbb{Q}}$, and $C$ has good reduction modulo $p$  for all $p \geq 5$.
  \end{lemma}
  
 \begin{proof}
It is enough to prove the result for an affine version (say $w=1$) of  $C$,
which is non-singular if and only if \\
  \begin{displaymath}
 \left[ \begin{array}{ccc} \frac{\partial F_1}{\partial x} & \frac{\partial F_1}{\partial y} &  \frac{\partial F_1}{\partial z}  \\ \frac{\partial F_2}{\partial x} & \frac{\partial F_2}{\partial y} & \frac{\partial F_2}{\partial z} \end{array} \right]
 \end{displaymath}
 has  full rank.  Consider the submatrix
  \begin{displaymath}
 \left[ \begin{array}{cc} \frac{\partial F_1}{\partial x} & \frac{\partial F_1}{\partial y}   \\ \frac{\partial F_2}{\partial x} & \frac{\partial F_2}{\partial y}  \end{array} \right]
= \left[ \begin{array}{cc} 2x & 2y\\
                              3x^2 & 3y^2
 \end{array} \right]
\end{displaymath}
 with determinant $6xy(x-y)=0$.
 If $p\neq 2,~3$ this implies $x=y$ or $x=0$ or $y=0$. Similarly considering the determinants 
 of other $2\times 2$ submatrices we get $x=y=z=0$ which is not a point on the curve. 
\end{proof}

\begin{theorem}\label{GENUS} The genus of  $C$  is $4$.
\end{theorem}

\begin{proof}
  Eliminating $z$ from affine equation of the curve above, we get 
 an irreducible plane curve with some singularities given by

 \begin{displaymath}
 f:=(x^3+y^3+1)^2+(x^2+y^2+1)^3=0.
 \end{displaymath}
 For singularites,
 $ \frac{\partial f}{\partial x}=0,~~ \frac{\partial f}{\partial y}=0$
 and solving these gives
\begin{displaymath}
 x^3+y^3+1=0, \\
  x^2+y^2+1=0.
 \end{displaymath}
There are six solutions to these equations, they are the points
$(a,(1+a^3)/(1+a^2))$ where $a$ is a root of
$2x^6+3x^4+2x^3+3x^2+2$.

Let $P=(a,b)$ be a singularity.
Then
$f(x+a,y+b)=((x+a)^3+(y+b)^3+1)^2+((x+a)^2+(y+b)^2+1)^3=F_0+F_1+\ldots+F_6$.
where $F_i$ is a form of degree $i$.

We have $F_0=F_1=0$ and
$F_2=(3a^2x+3b^2y)^2 \neq 0$ so the multiplicity of $f$ at $P$ is $2$.
Since $P$ was arbitrary $v_{P_i}=Mult_{P_i}C=2$.
 By the genus formula \cite{Abhyankar}, 148,
\[
genus(C)=\frac{(n-1)(n-2)}{2}-\sum_{i=1}^{6} \frac{(v_{P_i})(v_{P_i}-1)}{2}=10-6=4.
\]
  \end{proof}
 We note for the record that the curve $C$ is absolutely irreducible over $\mathbb{Q}$ and $\mathbb{F}_p$ for $p\geq 5$, but we will not use it in this paper.

\section{Genus 1 and Genus 0 Quotients}\label{g1quots}

We discuss the affine model of $C$ with $w=1$.

\begin{theorem}The quotient curve $C/(1,2)$ is an elliptic curve over $\overline{\mathbb{Q}}$.
\end{theorem}
\begin{proof} We have $(1,2):=\left\{ \begin{array}{rcl}
                          x\longmapsto y \\
                          y\longmapsto x.\\
                          \end{array} \right. $\\
 
This can be done directly but we proceed to use Algorithm \ref{A1}.
The fundamental invariants are  $a:=x+y,~b:=z,~ c:=x^2 + y^2$.
 The Groebner basis is found with MAGMA to be
 $G_0= \{x+y-a,~y^{2}-ya+ \frac{1}{2}a^{2}- \frac{1}{2}c,~z - b\}$
and $G_{1}:=\{x+y-a,~y^2-ya+ \frac{1}{2}a^{2}- \frac{1}{2}c,~z-b,~a^3-3ac+2bc+2b-2,~b^2+c+1\}$.
which gives $G_{2}=\{a^3-3ac + 2bc + 2b-2,b^2 + c + 1\}$.
Here $G_{2}$ is the defining ideal for the $K(C/(1,2))$. 
Substituting the second equation of $G_2$ into the first gives $a^3+3ab^{2}+ 3a-2b^3-2=0$ 
which on homegenization gives $a^3+3ab^{2}+ 3aw^2-2b^3-2w^3=0$.
Under the bi-rational subsitution $x=3a,~y=9b,~z=-a+b+w$ and putting $z=1$ we get
 $y^2 - 3xy - 9y = x^3 - \frac{27}{2}x - 27$.
\end{proof}

    \begin{theorem} 
  The quotient curve $C/(1,2,3,4)$ is an elliptic curve.
  \end{theorem}
  \begin{proof} 
  This time the Gr\"obner basis is too large to include.
  MAGMA gives $C/(1,2,3,4)$ as an elliptic curve defined by
    $$y^2 - 96xy + 110592y = x^3 + 3456x^2 + 14598144x -5718933504$$ over $\overline{\mathbb{Q}}$.\\
  \end{proof}

\begin{theorem}\label{TH1} 
The quotient curves $C/(1,2)$ and $C/(1,2,3,4)$ are isomorphic over $\overline{\mathbb{Q}}$.
\end{theorem}
\begin{proof}
Both the quotient curves have the same j-invariant equal to -36. Hence they are isomorphic over $\overline{\mathbb{Q}}$ . In fact, the isomorphism is defined  over $\mathbb{Q}$ and is given by\\
\begin{displaymath} 
\phi:C/(1,2)\rightarrow C/(1,2,3,4),~~~~~ (x,y)\rightarrow (1024x - 1152,~32768y - 2580481).
\end{displaymath}
Their  Weierstrass form is given by\\
\begin{displaymath}
y^2 = x^3 - 27x - 378.
\end{displaymath}
\end{proof}

\begin{remark}
MAGMA tells us that the endomorphism ring of $C/(1,2)$ 
is $\mathbb{Z}$, and so this elliptic curve does not have 
complex multiplication.
\end{remark}

\begin{theorem}\label{genus0}
The genus of the quotient curve $C/S_4$  is $0$.
\end{theorem}
\begin{proof}
 The fundamental invariants are $a:=x+y + z+ w,~ b:=x^2 + y^2 + z^2 + w^2,~c:=x^3 + y^3 + z^3 + w^3,~d:=x^4 + y^4 + z^4 + w^4.$\\
The Groebner Basis $G_0= \{x + y + z + w - a,y^2 + yz + yw - ya + z^2 + zw - za + w^2 - wa + 1/2a^2 - \frac{1}{2}b,z^3 + z^2w - z^2a + zw^2 - zwa + \frac{1}{2}za^2 - \frac{1}{2}zb + w^3 - w^2a + 1/2wa^2 - \frac{1}{2}wb - \frac{1}{6}a^3 + \frac{1}{2}ab - \frac{1}{3}c, w^4 - w^3a + 1/2w^2a^2 - \frac{1}{2}w^2b - \frac{1}{6}wa^3 + \frac{1}{2}wab - \frac{1}{3}wc + \frac{1}{24}a^4 - \frac{1}{4}a^2b + \frac{1}{3}ac + \frac{1}{8}b^2 - \frac{1}{4}d\}.$\\
$G_{1}=\{x + y + z + w - a,y^2 + yz + yw - ya + z^2 + zw - za + w^2 - wa + \frac{1}{2}a^2,z^3 + z^2w - z^2a + zw^2 - zwa + \frac{1}{2}za^2 + w^3 - w^2a + \frac{1}{2}wa^2 - \frac{1}{6}a^3, w^4 - w^3a +\frac{1}{2}w^2a^2 - \frac{1}{6}wa^3 + \frac{1}{24}a^4 - \frac{1}{4}d, b,c\}$,\\
which gives $G_2=\{b,c\}$.
The equations from $G_2$ define a projective line as asserted.\\
\end{proof}

\section{Genus 2 Quotient}\label{g2quots}

     \begin{theorem}The quotient curve $C/(1,2,3)$ is a hyperelliptic curve of genus $2$.
  \end{theorem}
   \begin{proof}We have $(1,2,3):=\left\{ \begin{array}{rcl}
                          x\longmapsto y \\
                          y\longmapsto z\\
                            z\longmapsto x.\\
                          \end{array} \right. $\\
 
 The fundamental invariants are $a:=x + y + z, b:=x^{2} + y^{2} + z^{2},c:=x^{3} + y^{3} + z^{3},d:=x^{2}z + xy^{2}+ yz^{2}.$\\

MAGMA gives the Groebner Basis $G_1$ to be $\{x + y + z - a,
    y^2 - \frac{4}{3}yd^5 + \frac{16}{3}yd^4 - \frac{28}{3}yd^3 + \frac{28}{3}yd^2 - \frac{17}{6}yd - \frac{1}{6}y + \frac{1}{24}z^2a^5d - \frac{1}{12}z^2a^5 - \frac{1}{12}z^2a^4d^2 + \frac{5}{2}4z^2a^4d - \frac{1}{12}z^2a^4 +
        \frac{1}{6}z^2a^3d^3 - \frac{1}{2}z^2a^3d^2 + \frac{3}{4}z^2a^3d - \frac{5}{6}z^2a^3 - \frac{1}{3}z^2a^2d^4 + \frac{7}{6}z^2a^2d^3 - 2z^2a^2d^2 + \frac{25}{12}z^2a^2d - \frac{1}{6}z^2a^2 +
        \frac{1}{2}z^2ad^3 - \frac{3}{2}z^2ad^2 + \frac{17}{8}z^2ad - \frac{5}{2}z^2a - z^2d^4 + \frac{7}{2}z^2d^3 - \frac{19}{4}z^2d^2 + \frac{25}{8}z^2d + \frac{3}{2}z^2 - \frac{1}{48}za^5 + \frac{1}{24}za^4d - \frac{1}{48}za^4
        - \frac{1}{12}za^3d^2 + \frac{1}{12}za^3d + \frac{1}{24}za^3 + \frac{1}{6}za^2d^3 - \frac{1}{4}za^2d^2 + \frac{5}{24}za^2 - \frac{1}{3}zad^4 + \frac{2}{3}zad^3 - \frac{1}{4}zad^2 - \frac{5}{12}zad + \frac{55}{48}za -
        \frac{2}{3}zd^5 + \frac{11}{3}zd^4 - \frac{49}{6}zd^3 + \frac{119}{12}zd^2 - \frac{145}{24}zd - \frac{1}{48}z - \frac{1}{48}a^5 + \frac{1}{24}a^4d - \frac{5}{48}a^4 - \frac{1}{12}a^3d^2 + \frac{1}{4}a^3d - \frac{7}{24}a^3 + \frac{1}{6}a^2d^3
        - \frac{7}{12}a^2d^2 + \frac{5}{6}a^2d - \frac{5}{8}a^2 + \frac{2}{3}ad^5 - 3ad^4 + 6ad^3 - \frac{83}{12}ad^2 + \frac{7}{2}ad - \frac{7}{16}a - \frac{1}{3}d^4 + \frac{7}{6}d^3 - \frac{7}{4}d^2 + \frac{29}{24}d + 3\frac{1}{48},
    yz + \frac{4}{3}yd^5 - \frac{16}{3}yd^4 + \frac{28}{3}yd^3 - \frac{28}{3}yd^2 + \frac{29}{6}yd - \frac{5}{6}y - \frac{1}{24}z^2a^5d + \frac{1}{12}z^2a^5 + \frac{1}{12}z^2a^4d^2 - \frac{5}{24}z^2a^4d + \frac{1}{12}z^2a^4 -
        \frac{1}{6}z^2a^3d^3 + \frac{1}{2}z^2a^3d^2 - \frac{3}{4}z^2a^3d + \frac{5}{6}z^2a^3 + \frac{1}{3}z^2a^2d^4 - \frac{7}{6}z^2a^2d^3 + 2z^2a^2d^2 - \frac{25}{12}z^2a^2d + \frac{2}{3}z^2a^2 -
        \frac{1}{2}z^2ad^3 + \frac{3}{2}z^2ad^2 - \frac{17}{8}z^2ad + \frac{5}{2}z^2a + z^2d^4 - \frac{7}{2}z^2d^3 + \frac{19}{4}z^2d^2 - \frac{25}{8}z^2d + z^2 + \frac{1}{48}za^5 - \frac{1}{24}za^4d + \frac{1}{48}za^4 +
        \frac{1}{12}za^3d^2 - \frac{1}{12}za^3d - \frac{1}{24}za^3 - \frac{1}{6}za^2d^3 + \frac{1}{4}za^2d^2 - \frac{5}{24}za^2 + \frac{1}{3}zad^4 - \frac{2}{3}zad^3 + \frac{1}{4}zad^2 + \frac{5}{12}zad - \frac{55}{48}za +
        \frac{2}{3}zd^5 - \frac{11}{3}zd^4 + \frac{49}{6}zd^3 - \frac{119}{12}zd^2 + \frac{169}{24}zd - \frac{95}{48}z + \frac{1}{48}a^5 - \frac{1}{24}a^4d + \frac{5}{48}a^4 + \frac{1}{12}a^3d^2 - \frac{1}{4}a^3d + \frac{7}{24}a^3 - \frac{1}{6}a^2d^3
        + \frac{7}{12}a^2d^2 - \frac{5}{6}a^2d + \frac{5}{8}a^2 - \frac{2}{3}ad^5 + 3ad^4 - 6ad^3 + \frac{83}{12}ad^2 - \frac{9}{2}ad + \frac{23}{16}a + \frac{1}{3}d^4 - \frac{7}{6}d^3 + \frac{7}{4}d^2 - \frac{29}{24}d + \frac{17}{48},
    ya + 2yd - y + \frac{1}{2}z^2a^2 + \frac{3}{2}z^2 + za + zd - 2z - \frac{1}{2}a^2 - ad + a + \frac{1}{2},
    yd^6 - 3yd^5 + 6yd^4 - 6yd^3 + \frac{33}{8}yd^2 - \frac{3}{4}yd + \frac{1}{8}y - \frac{1}{32}z^2a^5d^2 + \frac{1}{32}z^2a^5d - \frac{1}{32}z^2a^5 + \frac{1}{16}z^2a^4d^3 - \frac{3}{32}z^2a^4d^2 +
        \frac{3}{32}z^2a^4d - \frac{1}{32}z^2a^4 - \frac{1}{8}z^2a^3d^4 + \frac{1}{4}z^2a^3d^3 - \frac{9}{16}z^2a^3d^2 + \frac{7}{16}z^2a^3d - \frac{5}{16}z^2a^3 + \frac{1}{4}z^2a^2d^5 - \frac{5}{8}z^2a^2d^4 +
        \frac{11}{8}z^2a^2d^3 - \frac{19}{16}z^2a^2d^2 + \frac{13}{16}z^2a^2d - \frac{1}{16}z^2a^2 - \frac{3}{8}z^2ad^4 + \frac{3}{4}z^2ad^3 - \frac{51}{32}z^2ad^2 + {45}{32}z^2ad - \frac{27}{32}z^2a +
        \frac{3}{4}z^2d^5 - \frac{15}{8}z^2d^4 + \frac{51}{16}z^2d^3 - \frac{69}{32}z^2d^2 + \frac{27}{32}z^2d + \frac{21}{32}z^2 + \frac{1}{64}za^5d + \frac{1}{64}za^5 - \frac{1}{32}za^4d^2 - \frac{1}{64}za^4d + \frac{1}{64}za^4 +
        \frac{1}{16}za^3d^3 + \frac{3}{32}za^3d + \frac{5}{32}za^3 - \frac{1}{8}za^2d^4 + \frac{1}{16}za^2d^3 - \frac{3}{16}za^2d^2 - \frac{11}{32}za^2d + \frac{1}{32}za^2 + \frac{1}{4}zad^5 - \frac{1}{4}zad^4 +
        \frac{7}{16}zad^3 + \frac{1}{2}zad^2 + \frac{1}{64}zad + \frac{29}{64}za + \frac{1}{2}zd^6 - \frac{9}{4}zd^5 + \frac{39}{8}zd^4 - \frac{105}{16}zd^3 + \frac{159}{32}zd^2 - \frac{129}{64}zd - \frac{11}{64}z + \frac{1}{64}a^5d -
        \frac{1}{64}a^5 - \frac{1}{32}a^4d^2 + \frac{3}{64}a^4d - \frac{1}{64}a^4 + \frac{1}{16}a^3d^3 - \frac{1}{8}a^3d^2 + \frac{7}{32}a^3d - \frac{5}{32}a^3 - \frac{1}{8}a^2d^4 + \frac{5}{16}a^2d^3 - \frac{9}{16}a^2d^2 + \frac{13}{32}a^2d -
        \frac{1}{32}a^2 - \frac{1}{2}6 + \frac{7}{4}ad^5 - \frac{15}{4}ad^4 + 7\frac{1}{16}ad^3 - \frac{55}{16}ad^2 + \frac{81}{64}ad - \frac{33}{64}a + \frac{1}{4}d^5 - \frac{5}{8}d^4 + \frac{19}{16}d^3 - \frac{23}{32}d^2 + \frac{19}{64}d + \frac{23}{64},
    z^3 - z^2a + \frac{1}{2}za^2 + \frac{1}{2}z - \frac{1}{6}a^3 - \frac{1}{2}a + \frac{1}{3},
    a^6 + 9a^4 - 8a^3 + 27a^2 + 24ad - 48a + 24d^2 - 24d + 27,
    b + 1,
    c + 1\}$.
    
Therefore by Algorithm 1 the defining set  for the
quotient curve is 
$G_{2}=\{a^6 + 9a^4 - 8a^3 + 27a^2 + 24ad - 48a + 24d^2 - 24d + 27,~b+1,~c+1\}$.
The change of variables $y:=-\frac{1}{2}a - \frac{3}{2}d,~x=\frac{1}{2}(-a+1)$ defines a hyperelliptic curve given by 
$y^2 + (x^3 + x^2)y = -x^6 + 4x^5 - 25x^4 + 36x^3 - 36x^2 + 18x - 6$.\\ 
\end{proof}

Let  $V$ be a curve of genus $2$ over a field $k$. Let $Jac(V)$ be its Jacobian. We say that $J$ is split over $k$ if $J$ is isogenous over $k$ to a product of elliptic curves $E_{1} \times E_{2}$. 
\begin{theorem}(Kuhn\cite{Kuhn})\label{thm:nnsplit}
Let $J$ be a Jacobian of a curve $V$ of genus $2$ over a field $k$ of characteristic different from 2. Suppose that $J$ is split. 
Then there are elliptic curves $E_1$, $E_2$ over $k$, and an integer $n>1$ such that
$E_1[n]$ and $E_2[n]$ are isomorphic as group schemes and $J$ is
isogenous to $E_1\times E_2$. Furthermore, the curve $V$ admits degree $n$ covers $V\to E_1$ and $V\to E_2$.
\end{theorem}
Whenever the above happens, we call $Jac(V)$ a $(n,n)$-split, and 
we say $Jac(V)$ is $(n,n)$-isogenous to $E_1 \times E_2$

 Let $I_2$, $I_4$, $I_6$, and $I_{10}$ (see \cite{Igusa}) of a genus $2$ curve $V$. We define absolute invariants,
\begin{displaymath}
i_1 = 144\frac{I_4}{I_{2}^{2}},~~i_2 = -1728\frac{{I_2I_4 - 3I_6}}{I_{2}^{3}},~~ i_3 = 486\frac{I_{10}}{I_{2}^{5}}.
\end{displaymath}
\begin{theorem}(Bruin-Doerksen)\label{Briun}
The absolute invariants $i_1,i_2,i_3$ of a genus $2$ curve with optimally
 $(4,4)$-split Jacobian satisfy an equation $L_4=0$.
\end{theorem}
The equation $L_4$ is to big to be reproduced here. See \cite{Nils,Equation} for proof and the equation $L_{4}$.

The projective equation of  $C/(1,2,3)$ is given by   $y^2z^4 + (x^3z^2 + x^2z^3)y = -x^6 + 4x^5z - 25x^4z^2 + 36x^3z^3 - 36x^2z^4 + 18xz^5 - 6z^6$.
Performing the birational subsitution  $(x,y,z)\rightarrow (x,x^3,x^2z+2y)$  and then putting $z=1$ gives the Weierstrass form\\
\begin{displaymath}
 y^2 = -3x^6 - 18x^5 - 99x^4 - 144x^3 - 144x^2 - 72x - 24.
\end{displaymath}

The Igusa invariants are $I_{2}=-138240$, $I_{4}= 234150912$, $I_{6}= -448888946688$, $I_{10}=-12999674453557248$. This gives the absolute invariants to be
$i_{1}=2823/1600$, $i_{2}=2597331/128000$,
 $i_{3}=6561/52428800000$. It can be checked using any computer algebra package
 that the absolute invariants satisfy the equation $L_4=0$.  Hence by theorem \ref{Briun} we strongly suspected that $Jac(C/(1,2,3))$ is a $(4,4)$- split.
We now prove this.
\begin{theorem} \label{Splits}
Let $E_1$ be the elliptic curve $C/(1,2)$, and let $E_2$ be the elliptic curve over $\overline{\mathbb{Q}}$ defined by
$y^2+xy=x^3-x^2-6x+8$. Then Jacobian of $C/(1,2,3)$ is $(4,4)$-isogenous to $E_1 \times E_2$.
\end{theorem}
\begin{proof}
We will use the method given in section 4 of \cite{Nils}. We first prove the existence of $(2,2)$ Richelot isogeny from $Jac(C/(1,2,3))$ to a certain $(2,2)$-split Jacobian, say $B$,  and then construct it.
The hyperelliptic curve $C/(1,2,3)$ is given by $y^2=-3F$ where  $F=x^6 + 6x^5 + 33x^4 + 48x^3 + 48x^2 + 24x + 8$. The splitting field of the polynomial $F$ has is the cubic number field $S$ defined by $x^3-3x+4$ and discriminant $d=-324$.
Therefore we get $F=F_1F_2F_3$ 
where 

$F_1=x^2 + \frac{1}{18}(a^4 - 3a^2)x + \frac{1}{18}(a^4 - 9a^2 - 18)$, \\
$F_2=x^2 + \frac{1}{36}(-a^5 - a^4 + 15a^3 + 3a^2 - 72a + 108)x + \frac{1}{72}(-a^5 - 2a^4 +15a^3 + 18a^2 - 108a - 72),$ \\
$F_3=x^2 + \frac{1}{36}(a^5 - a^4 - 15a^3 + 3a^2 + 72a + 108)x + \frac{1}{72}(a^5 - 2a^4 - 15a^3+ 18a^2 + 108a - 72)$, \\
where $a$ satisfies $x^6-18x^4+81u^2+324$.
If we write $F_j(X)=q_2 X^2+ q_{1,j} X+ q_{0,j}$\\
then we define\\
\begin{equation}\nonumber
\delta:=\det
\begin{pmatrix}
 q_{0,1}& q_{1,1} & q_{2}\\
 q_{0,2}& q_{1,2} & q_{2}\\
 q_{0,3}& q_{1,3} & q_{2}\\
\end{pmatrix} = 2a^3-18a.
\end{equation}
Since $\delta \neq 0$,  $B$ is the Jacobian of genus $2$ curve.
We will find the equation of a  curve $\tilde{C}$ such that  $Jac(\tilde{C})=B$.
We define for  $(i,j,k)=(1,2,3),(2,3,1),(3,1,2)$
$$G_i(X)=\delta^{-1}\det
\begin{pmatrix}
  \frac{d}{dX}F_j(X)& \frac{d}{dX}F_k(X)\\
  F_j(X)            & F_k(X)
\end{pmatrix}.
$$
Hence,
$ G_1=\frac{1}{648}(-a^4 + 27a^2 - 108)x^2 + \frac{1}{216}(-a^4 + 15a^2 - 36)x + \frac{1}{162}(-a^4 +18a^2 - 108)$,\\
 $G_2=\frac{1}{1296}(-a^5 + a^4 + 15a^3 - 27a^2 + 108)x^2 + \frac{1}{432}(a^4 - 15a^2 + 36a + 36)x+ \frac{1}{1296}(-a^5 + 4a^4 + 15a^3 - 72a^2 + 108a - 216)$,\\
and  $G_3=\frac{1}{1296}(a^5 + a^4 - 15a^3 - 27a^2 + 108)x^2 + \frac{1}{432}(a^4 - 15a^2 - 36a + 36)x+ \frac{1}{1296}(a^5 + 4a^4 - 15a^3 - 72a^2 - 108a - 216)$.\\
Hence we get $\tilde{C}$ defined by $y^2=dG_1G_2G_3$,  i.e.,
$y^2=f(x)=2x^6 + 6x^5 + 15x^4 + 18x^3 + 15x^2 + 6x + 2$.

To see that $\tilde{C}$ is a  $(2,2)$ cover of elliptic curves (say $E_1$ and $E_2$) make the change of variable $x \rightarrow \frac{x+1}{x-1}$ 
we get  $f=\frac{64x^6 + 36x^4 + 24x^2 + 4}{x^6 - 6x^5 + 15x^4 - 20x^3 + 15x^2 - 6x + 1}$
Put $u=x^2$ in the numerator of $f$ we get $y^2=64u^3+36u^2+24u+4$  as one of the elliptic curve say $E_1$, the covering map is given by $(x,y) \rightarrow (x^2, y)$.
Then by  \cite{Cassels} p.155, $E_2$ is given by $y^2=4u^3+24u^2+3 6u+64$ and the covering map is given by $(x,y)\rightarrow (\frac{1}{x^2}, \frac{y}{x^3})$.

\end{proof}

\section{Arithmetic Properties of The Quotients that arise from the Splitting}\label{reductions}

When we write $L(J_{C})$ we mean the L-polynomial of the reduction modulo $p$ 
of the Jacobian of $C$.  
The Jacobian has good reduction because $C$ does.

\begin{theorem}\label{lpo} 
Over $\mathbb{F}_{p},~ p\geq 5$,  
$$L(J_{C}) =  L(J_{C/(1,2)}) \times L(J_{C/(1,2,3)}) \times L(J_{C/(1,2,3,4)}).$$
\end{theorem}
 \begin{proof}
 First we will work over $\overline{\mathbb{Q}}$.
 By Lemma \ref{GENUS}, C is a non-singular over $\overline{\mathbb{Q}}$.
 The automorphism group of $C$ contains the symmetric group on $4$ symbols i.e; $S_4\leq Aut(C)$.
By the symmetry of the curve, one can conclude that $\epsilon_{<\sigma>}(J_C) \sim  \epsilon_{<\tau > }(J_C)$ if $\sigma$ and $\tau$ are of same cycle type.

In $S_4$ there are
\begin{enumerate}
\item six cyclic subgroups of order 2, say $\sigma_{1}, \ldots , \sigma_{6}$.
\item four cyclic subgroups of order 3, say $\tau_{1}, \ldots , \tau_{4}$.
\item three cyclic subgroups of order $4$, say $F_{1} \ldots , F_{3}$.
\end{enumerate}
We have the following idempotent relation in $S_4$, namely 
\begin{displaymath}
12\epsilon_{id}+ 24\epsilon_{{S_{4}}}=2[\epsilon_{\sigma_{1}}+\ldots \epsilon_{\sigma_6}]+3[\epsilon_{\tau_{1}}+\ldots \epsilon_{\tau_4}]+4[\epsilon_{F_1}+ \dots +\epsilon_{F_3}],
\end{displaymath}
which  by Kani-Rosen decomposition implies
\begin{displaymath}
\epsilon_{id}(J_C)^{12} \times \epsilon_{{S_{4}}}(J_C)^{24} \sim \epsilon_{\sigma_{1}}(J_C)^2\times \ldots \times \epsilon_{\sigma_6}(J_C)^2 \times \epsilon_{\tau_{1}}(J_C)^3 \times \ldots \epsilon_{\tau_4}(J_C)^3 \times \epsilon_{F_1}(J_C)^4\times \dots \times \epsilon_{F_3}(J_C)^4.
\end{displaymath}
This gives the following isogeny over $\overline{\mathbb{Q}}$
\begin{displaymath}
J_{C}^{12} \times   J_{C/S_4}^{24}\sim  J_{C/(1,2)}^{12} \times J_{C/(1,2,3)}^{12} \times J_{C/(1,2,3,4)}^{12}.
\end{displaymath}
Reducing this isogeny modulo $p\geq 5$ gives an isogeny over $\mathbb{F}_p$, which means 
that the $L$-polynomials satisfy
 \[
 L(J_{C})^{12}\times L(J_{C/S_4})^{12} =  L(J_{C/(1,2)})^{12} \times L(J_{C/(1,2,3)})^{12}\times L(J_{C/(1,2,3,4)})^{12}.
 \]
 But $\dim (J_{C/S_4})=0$ from Theorem \ref{genus0}
 which implies that $L((J_{C/S_4}))=1$.
 By taking the $12th$-root on both sides we get the result.
\end{proof}

\begin{corollary}\label{c1}
The modulo  $p$ reductions of $J_C$ and 
$J_{C/(1,2)} \times J_{C/(1,2,3)} \times J_{C/(1,2,3,4)}$ are isogenous
over $\mathbb{F}_p$.
\end{corollary}

\begin{proof} 
By Theorem \ref{lpo} and  Tate's theorem, we can conclude that there is an $\mathbb{F}_p$-isogeny 
between (the mod $p$ reductions of) $J_C$ and 
$J_{C/(1,2)} \times J_{C/(1,2,3)} \times J_{C/(1,2,3,4)}$.
\end{proof}

\bigskip

From the equations computed earlier of the  quotient curves defined over $ \overline{\mathbb{Q}}$,
we tabulate their L-polynomials
 over $\mathbb{F}_{p}$ for few primes $p$ below.
\bigskip

\begin{tabular}{|l|l|l|l|l|}
\hline
 $p$& $ L(J_{C/(1,2)})$ & $L(J_{C/(1,2,3)})$ & $L(J_{C/(1,2,3,4)})$& $p$-rank  \\ \hline
$5$&   $(5t^2-t+1)$    & $(5t^2-t+1)(5t^2+3t+1)   $&$(5t^2-t+1)    $&$4$\\  \hline
$7$&    $7t^2 + 1 $    & $(7t^2+1)(7t^2+4t+1)     $&$7t^2 + 1      $&$1$\\  \hline
$11$  & $11t^2-4t+1$   & $(11t^2-4t+1)(11t^2 + 1) $&$11t^2 - 4t + 1$&$3$\\  \hline
$13$  & $13t^2+5t+1$   & $(13t^2+t+1)(13t^2+5t+1) $&$13t^2 + 5t + 1$&$4$\\  \hline
$17$  & $17t^2-5t+1$   & $(17t^2-5t+1)(17t^2+3t+1)$&$17t^2 - 5t + 1$&$4$\\  \hline
$19$  & $19t^2-8t+1$   & $(19t^2-8t+1)(19t^2+4t+1)$&$19t^2 - 8t + 1$ &$4$\\  \hline
$23$  & $23t^2-4t+1$   & $(23t^2-4t+1)(23t^2+1)   $&$23t^2 - 4t + 1$&$3$\\ \hline
$29$  & $29t^2+3t+1$   & $(29t^2-9t+1)(29t^2+3t+1)$&$29t^2 + 3t + 1$&$4$\\ \hline
$31$  & $31t^2+4t+1$   & $(31t^2+4t+1)^2          $&$31t^2 + 4t + 1$&$4$\\ \hline
$37$  & $37t^2-3t+1$   & $(37t^2+t+1)(37t^2-3t+1) $&$37t^2-3t + 1  $&$4$\\ \hline
$37$  & $37t^2-3t + 1$ &  $(37t^2 - 3t + 1)(37t^2 + t + 1)$ & $37t^2 - 3t + 1$ &$4$\\  \hline
$41$  & $ 41t^2 - 6t +1$ &  $(41t^2 - 6t + 1)^2$ & $41t^2 - 6t + 1$&$4$ \\  \hline
$43$  & $43t^2 - 4t + 1$ &  $(43t^2 - 8t + 1)(43t^2 - 4t + 1)$ & $43t^2 - 4t + 1$ &$4$\\  \hline
$47$  & $47t^2 - 12t + 1$&$(47t^2 - 12t + 1)(47t^2 + 12t + 1)$&$47t^2 - 12t + 1 $&$4$ \\  \hline
$53$  & $53t^2 - 10t + 1$&$(53t^2 - 10t + 1)(53t^2 + 6t + 1)$&$53t^2 - 10t + 1$&$4$ \\  \hline
$59$  & $59t^2 + 8t + 1$ &$(59t^2 + 1)(59t^2 + 8t + 1)$& $59t^2 + 8t + 1$ &$3$\\  \hline
$61$  & $61t^2 + 5t + 1$&$(61t^2 + t + 1)(61t^2 + 5t + 1)$&$61t^2 + 5t + 1$&$4$ \\  \hline
$67$  & $67t^2 - 8t + 1$&$(67t^2 - 8t + 1)(67t^2 + 4t + 1)$&$67t^2 - 8t + 1$&$4$ \\  \hline
$71$  & $71t^2 + 16t + 1$&$(71t^2 + 12t + 1)(71t^2 + 16t + 1)$&$71t^2 + 16t + 1$ &$4$\\  \hline
$73$  & $73t^2 + 5t + 1$&$(73t^2 - 11t + 1)(73t^2 + 5t + 1)$&$73t^2 + 5t + 1$&$4$ \\  \hline
$79$  & $79t^2 - 4t + 1$&$(79t^2 - 4t + 1)(79t^2 + 16t + 1)$&$79t^2 - 4t + 1$ &$4$\\  \hline
$83$  & $83t^2 + 4t + 1$&$(83t^2 + 4t + 1)(83t^2 + 12t + 1)$&$83t^2 + 4t + 1$ &$4$\\  \hline
$89$  & $89t^2 + 3t + 1$&$(89t^2 + 3t + 1)^2$&$89t^2 + 3t + 1$&$4$ \\  \hline
$97$  & $97t^2 - 2t + 1$&$(97t^2 - 2t + 1)^2$ & $97t^2 - 2t + 1$&$4$ \\  \hline
$101$ & $101t^2 + 6t + 1$&$(101t^2 + 6t + 1)^2$ & $101t^2 + 6t + 1$&$4$ \\  \hline
$103$ & $103t^2 - 4t + 1$&$(103t^2 - 4t + 1)(103t^2 + 4t + 1)$ & $103t^2 - 4t + 1$&$4$ \\  \hline

\end{tabular}\\

  \begin{remark}
Theorem \ref{Splits} implies that
over $\mathbb{F}_{p}$, for all $p \geq 5$,  the
Jacobian of $C/(1,2,3)$ (reduced modulo $p$) is isogenous to the 
product of $\overline{E_1}$ and $\overline{E_2}$, 
where $\overline{E}$ denotes the reduction of $E$ modulo  $p$. 
By Corollary \ref{c1}  we get  that
\begin{displaymath}
\overline{J_C}\sim \overline{E_{1}}^3 \times \overline{E_{2}}.
\end{displaymath}
\end{remark}

 \begin{remark}
For $p=31,41, 89, 97,101$, the table shows that
$\overline{E_1}$ is isogenous to $\overline{E_2}$, and so
\begin{displaymath}
\overline{J_C}\sim \overline{E_{1}}^4.
\end{displaymath}
It would be interesting to know how often this happens.
If $\overline{J_C} \sim \overline{E_1} \times \overline{E_2}$ has L-polynomial
$q^2 t^4+qa_1t^3+a_2t^2+a_1t+1$ then $\overline{E_1}$ is isogenous to 
$\overline{E_2}$ if and only if $a_1^2 -4a_2+8q=0$, but we cannot see
how to predict  for which $p$ this will happen.
\end{remark}

\section{Number of Rational Points}\label{pts}

It would be nice to have a formula for the number of rational points 
on the reduction of $C$ modulo $p$.  However, our investigations show that
this is unlikely.

The table below gives the number of rational points of $C$ for some values of $q$ 
and the Hasse -Weil-Serre bounds.\\

\begin{tabular}{|l|l|l|l|}
\hline
$p$&Lower bound& $  \# C(\mathbb{F}_q)$ & Upper bound\\ \hline
5&-10&6&22 \\ \hline
7&-12&12&28 \\ \hline
11&-12&0&36 \\ \hline
13&-14&30&42 \\ \hline
17&-14&6&50 \\ \hline 
19&-12&0&52\\ \hline
23&-12&12&60\\ \hline
29&-10&30&70\\ \hline
31&-12&48&76\\ \hline
37&-10&30&86\\ \hline
41&-6&18&90\\ \hline
43&-8&24&96\\ \hline
47&-4&24&100\\ \hline
53&-2&30&110\\ \hline
59&0&84&120\\ \hline
61&2&78&122\\ \hline
67&4&48&132\\ \hline
71&8&132&136\\ \hline
73&6&78&142\\ \hline
79&12&84&148\\ \hline
83&12&108&156\\ \hline
89&18&102&162\\ \hline
97&22&90&174\\ \hline
101&22&126&182\\ \hline
103&24&96&184\\ \hline
\end{tabular}\\
 
\begin{remark}
\begin{enumerate}
\item For $p=71$ the defect is only $4$, which is very good.
We do not know if this is best possible   for genus $4$ over that field.
\item For $p= 11,19$ the curve $C$ has no points.    
\end{enumerate}
\end{remark}

Because we have taken explicit exponents $a=2, b=3$ in this article,
many things can be computed explicitly.  
In future work we will investigate other values of $a$ and $b$.

\end{document}